\documentclass[11pt]{amsart}
\usepackage{fullpage,verbatim,amssymb,mathtools}
\usepackage[active]{srcltx}
\usepackage{hyperref}
\usepackage{enumerate}
\usepackage[usenames,dvipsnames]{color}
\usepackage{bm}
\usepackage{bbm}

\definecolor{blue}{rgb}{0,0,1}
\definecolor{red}{rgb}{1,0,0}
\definecolor{green}{rgb}{0,.6,.2}
\definecolor{purple}{rgb}{1,0,1}
\long\def\red#1\endred{\textcolor{red}{#1}}
\long\def\blue#1\endblue{\textcolor{blue}{#1}}
\long\def\purple#1\endpurple{\textcolor{purple}{ #1}}
\long\def\green#1\endgreen{\textcolor{green}{#1}}

\DeclareMathOperator{\ord}{ord}

\DeclareMathOperator{\Res}{Res}
\DeclareMathOperator{\tr}{tr}
\DeclareMathOperator{\lcm}{lcm}

\newcommand{\C}{\mathbb{C}}
\newcommand{\R}{\mathbb{R}}
\newcommand{\Z}{\mathbb{Z}}
\newcommand{\Q}{\mathbb{Q}}
\newtheorem{theorem}{Theorem}
\newtheorem{lemma}[theorem]{Lemma}
\newtheorem{prop}[theorem]{Proposition}

\theoremstyle{remark}
\newtheorem{remarks}[theorem]{Remarks}

\numberwithin{theorem}{section}
\numberwithin{equation}{section}

\title{An extension of Venkatesh's converse theorem to the Selberg class}
\author{Andrew R. Booker}
\author{Michael Farmer}
\author{Min Lee}
\address{School of Mathematics, University of Bristol, Woodland Road, Bristol, BS8 1UG}
\thanks{M.~L.\ was supported by
Royal Society University Research Fellowship ``Automorphic forms,
L-functions and trace formulas''.}

\date{}

\begin{document}
\maketitle

\begin{abstract}
We extend Venkatesh's proof of the converse theorem for classical holomorphic modular forms to arbitrary level and character. The method of proof, via the Petersson trace formula, allows us to treat arbitrary degree $2$ gamma factors of Selberg class type.
\end{abstract}

\section{Introduction}
Venkatesh, in his thesis \cite{akshay}, gave a new proof of the classical converse theorem for modular forms of level $1$, in the context of Langlands' ``Beyond Endoscopy''. The key analytic input is Voronoi summation, or equivalently the functional equation for additive twists.
More precisely, given a modular form $f(z)=\sum_{n=1}^\infty f_nn^{\frac{k-1}{2}}e(nz)$
of weight $k$, level $N$, and nebentypus character $\chi$, we define the complete additive twist
\[
\Lambda_f(s,\alpha)=\Gamma_\C\bigl(s+\tfrac{k-1}{2}\bigr)\sum_{n=1}^\infty\frac{f_ne(n\alpha)}{n^s}
\quad\text{ for }\alpha\in\Q,
\]
where $e(z)=e^{2\pi iz}$ and $\Gamma_\C(s) = 2(2\pi)^{-s}\Gamma(s)$.
Then, for any $q\in N\Z_{>0}$ and any $a,\bar{a}\in\Z$ with $a\bar{a}\equiv1\pmod{q}$,
$\Lambda_f\bigl(s,\tfrac{a}{q}\bigr)$ and $\Lambda_f\bigl(s,-\tfrac{\bar{a}}{q})$ continue to entire functions and satisfy the functional equation \cite[(A.10)]{kmv02}
\[
\Lambda_f\bigl(s,\tfrac{a}{q}\bigr)= i^k\chi(\bar{a})q^{1-2s}\Lambda_f\bigl(1-s,-\tfrac{\bar{a}}{q}\bigr).
\]

In this paper we generalize Venkatesh's proof to forms of arbitrary level and character.
Our precise result is the following:
\begin{theorem}\label{Main theorem}
Consider $\{f_n\}_{n\ge1}$, $N$, $\chi$, $\omega$, and $\gamma(s)$ with the following properties:
\begin{enumerate}
    \item $\{f_n\}_{n\ge1}$ is a sequence of complex numbers
    such that $\sum_{n=1}^\infty f_nn^{-s}$ converges absolutely for $\Re(s)>1$;
    \item $N$ is a natural number and $\chi$ is a Dirichlet character modulo $N$;
    \item $\omega$ is a nonzero complex number;
    \item $\gamma(s)=Q^s\prod_{j=1}^r\Gamma(\lambda_js+\mu_j)$ for some numbers $Q,\lambda_j\in\R_{>0}$ and $\mu_j\in\C$ with $\Re(\mu_j)>-\frac12\lambda_j$ and $\sum_{j=1}^r\lambda_j=1$.
\end{enumerate}
Given any $\alpha\in\Q$, define the complete twisted $L$-function
\[
\Lambda_f(s,\alpha)=\gamma(s)\sum_{n=1}^\infty\frac{f_ne(n\alpha)}{n^s}.
\]
Suppose that for every $q\in N\Z_{>0}$ and every pair $a,\bar{a}\in\Z$ with $a\bar{a}\equiv1\pmod q$, 
$\Lambda_f\bigl(s,\tfrac{a}{q}\bigr)$ and $\Lambda_f\bigl(s,-\tfrac{\bar{a}}{q}\bigr)$ continue to entire functions of finite order and satisfy the functional equation
\begin{equation}{\label{functional equation}}
\Lambda_f\bigl(s,\tfrac{a}{q}\bigr) = \omega\chi(\bar{a})q^{1-2s}\Lambda_f\bigl(1-s,-\tfrac{\bar{a}}{q}\bigr).
\end{equation}
Then there exists $k\in\Z_{>0}$ such that $f(z):=\sum_{n=1}^{\infty}f_nn^{\frac{k-1}{2}}e(nz)$
is a modular form of weight $k$, level $N$, and nebentypus character $\chi$.
\end{theorem}

\begin{remarks}\
\begin{enumerate}
    \item One feature of the argument is that it admits a generalization to gamma factors of Selberg class type. In this way, our result can be viewed as a converse theorem for degree $2$ elements of the Selberg class, albeit with infinitely many functional equations. Recently, Kaczorowski and Perelli \cite{kp20} have classified the elements of the Selberg class of conductor $1$ without the need for any twists. Very little is known for higher conductor, however, and our result is the first that we are aware of to consider both arbitrary level and degree $2$ gamma factor.
    \item For $k>1$, it is enough to assume that the finite $L$-functions $L_f\bigl(s,\tfrac{a}{q}\bigr)=\Lambda_f\bigl(s,\tfrac{a}{q}\bigr)/\gamma(s)$ have analytic continuation to $\C$, and in this case we can also conclude that $f$ is cuspidal. When $k=1$, we need to know that $\Lambda_f\bigl(s,\tfrac{a}{q}\bigr)$ is analytic at $s=0$, and there are noncuspidal examples satisfying all of the hypotheses of Theorem~\ref{Main theorem}.
    \item If we suppose that $L_f(s,1)$ lies in the Selberg class then we can combine the transformation formula in \cite[Theorem~2]{kp15} with the Vandermonde argument in \cite[Lemma~2.4]{bk14} to constrain the possible poles of $\Lambda_f\bigl(s,\tfrac{a}{q}\bigr)$. In this way, it is likely possible to prove a result that allows the twisted $L$-functions $\Lambda_f\bigl(s,\tfrac{a}{q}\bigr)$ to have arbitrary poles inside the critical strip, but we have not pursued this.
    \item Using the Bruggeman--Kuznetsov trace formula and the method of \cite{m21}, it is likely possible to prove a similar converse theorem for Maass forms.
\end{enumerate}
\end{remarks}

\section{Lemmas}
We begin with some preparatory lemmas.
 
\begin{lemma}{\label{classifying gamma factors}}
Let $\gamma(s)=\prod_{j=1}^r\Gamma(\lambda_js+\mu_j)$, where $\lambda_j\in\R_{>0}$, $\mu_j\in\C$. If $\gamma(s)$ has poles at all but finitely many negative integers then $\sum_{j=1}^r\lambda_j\ge1$.
\end{lemma}
\begin{proof}
Since the poles of $\Gamma(\lambda_js+\mu_j)$ are spaced $\lambda_j^{-1}$ apart,
the number of poles of $\gamma(s)$ in $\Z\cap[-T,0)$ for large $T>0$ is at most $T\sum_{j=1}^r\lambda_j+O(1)$. If all but finitely many negative integers are poles of $\gamma(s)$ then this count is at least $T+O(1)$. The conclusion follows on taking $T\to\infty$.
\end{proof}

\begin{lemma}{\label{no pole lemma}}
Let $\gamma(s)$ be as in the statement of Theorem~\ref{Main theorem}, and suppose that $\gamma(s)$ has poles at all but finitely many negative integers. Then $\gamma(s)$ is of the form $cP(s)H^s\Gamma_\C(s)$, where $c,H\in\R_{>0}$ and $P$ is a monic polynomial whose roots are distinct non-positive integers.
\end{lemma}
\begin{proof}
Recall that $\gamma(s)=Q^s\prod_{j=1}^r\Gamma(\lambda_js+\mu_j)$. By Lemma~\ref{classifying gamma factors} each factor in the product must contribute infinitely many poles in the negative integers; in particular $\lambda_j,\mu_j\in\mathbb{Q}$ for each $j$.
Let $\lambda_j=a_j/q_j$ in lowest terms. If $a_j>1$ then we can consider
$\tilde{\gamma}(s)=\Gamma(\frac{s-m_j}{q_j})
\prod_{i\ne j}\Gamma(\lambda_i s + \mu_i)$,  where $m_j$ is the negative integral pole of $\Gamma(\lambda_j s + \mu_j)$ of smallest absolute value. 
Then $\tilde{\gamma}(s)$ also has poles at all but finitely many negative integers, contradicting Lemma~\ref{classifying gamma factors}.

Hence we have $a_j=1$, so $\gamma(s)=Q^s\prod_{j=1}^r\Gamma(\frac{s+\nu_j }{q_j})$, where $\nu_j=\mu_jq_j\in\mathbb{Z}_{\ge0}$. Let $q=\lcm\{q_1,\ldots,q_j\}$ and write $\Gamma(\frac{s+\nu_j }{q_j})=\Gamma(\frac{q}{q_j}\cdot\frac{s+\nu_j}{q})$. By the Gauss multiplication formula, we have 
\begin{equation*}
    \gamma(s)=\bigl(\sqrt{2\pi}\bigr)^{r-n}Q^s\prod_{j=1}^r\left(\frac{q}{q_j}    \right)^{\frac{s+\nu_j}{q_j}-\frac12}\prod_{i=0}^{q/q_j-1}
    \Gamma\!\left(\frac{s+\nu_j+iq_j}{q}\right),
\end{equation*}
which we can rewrite in the form
$
cH^s\prod_{j=1}^q\Gamma\!\left(\frac{s+\nu_j'}{q}\right),
$
for some $c,H\in\R_{>0}$ and $\nu_j'\in\Z_{\ge0}$. Moreover, the $\nu_j'$ must run through a complete set of representatives for the residue classes mod $q$. Replacing each $\nu_j'$ by its mod $q$ reduction $\nu_j''\in\{0,\ldots,q-1\}$ and using the recurrence formula to relate $\Gamma(\frac{s+\nu_j'}{q})$ and $\Gamma(\frac{s+\nu_j''}{q})$, we get
$
c'P(s)H^s\prod_{j=1}^q\Gamma\!\left(\frac{s+\nu_j''}{q}\right),
$
with $P$ as described in the statement of the lemma. Finally, applying the Gauss multiplication formula once more and inserting a factor of $2(2\pi)^{-s}$, we arrive at the claimed form.
\end{proof}

The final lemma that we need is a special case of \cite[Lemma~4.1]{m21}. 
\begin{lemma}\label{lem:Ramanujansum_series}
Let $r_q(n)=\sum_{\substack{a\bmod q\\(a,q)=1}}e(na/q)$ be the Ramanujan sum. Then, for $n,N\ge\Z_{>0}$ and $\Re(s)>1$, we have
\[
\sum_{\substack{q\ge1\\N\mid q}} \frac{r_q(n)}{q^{2s}}
= \begin{cases} 
\frac{\sigma_{1-2s}(n; N)}{\zeta^{(N)}(2s)} & \text{ when } n\neq 0, \\
N^{1-2s} \prod_{p\mid N} (1-p^{-1}) \frac{\zeta(2s-1)}{\zeta^{(N)}(2s)} & \text{ when } n=0. 
\end{cases} 
\]
Here, when $\frac{N}{\prod_{p\mid N}p}\mid n$, 
\begin{equation}\label{e:sigma_nN}
\sigma_{s}(n; N)=
\prod_{\substack{p\mid n\\p\nmid N}}
\frac{p^{(\ord_p(n)+1)s}-1}{p^{s}-1}
\cdot\prod_{p\mid N}
\left(\frac{(1-p^{s-1})p^{(\ord_p(n)+1)s}-(1-p^{-1})p^{\ord_p(N)s}}{p^{s}-1}
\right).
\end{equation}
Otherwise $\sigma_{s}(n; N)=0$.
\end{lemma}

\section{Proof of Theorem~\ref{Main theorem}}
Let $S_k(N,\chi)$ denote the space of holomorphic modular forms of weight $k$, level $N$ and nebentypus character $\chi$, and let $H_k(N,\chi)$ be an orthonormal basis for $S_k(N,\chi)$.  
Each $g\in H_k(N,\chi)$ has a Fourier expansion of the form  
\[
g(z) = \sum_{n=1}^\infty \rho_g(n)n^{\frac{k-1}{2}}e(nz)
\]
for some coefficients $\rho_g(n)\in\C$.
An application of the Petersson formula gives the following formula found in  \cite[Corollary~14.23]{ik04}: for positive integers $n$ and $m$, we have 
\begin{equation}\label{e:Petersson}
  \frac{\Gamma(k-1)}{( 4 \pi)^{k-1}}\sum_{g\in H_k(N,\chi)}\rho_g(n)\overline{\rho_g(m)} 
= \delta_{n=m} + 2\pi i^{-k} \sum_{\substack{q\geq 1\\ N\mid q}} \frac{S_{\chi}(m, n; q)}{q} J_{k-1}\!\left(\frac{4\pi \sqrt{mn}}{q}\right), 
\end{equation}
 where 
\[
S_{\chi}(m, n; q) = \sum_{\substack{a\bmod{q}\\ \gcd(a, q)=1}} \chi(a) e\!\left(\frac{ma+n\bar{a}}{q}\right)
\]
is the twisted Kloosterman sum 
and $J_{k-1}(y)$ is the classical $J$-Bessel function.

Fix a choice of data $\{f_n\}_{n\ge1}$, $N$, $\chi$, $\omega$, and $\gamma(s)$ satisfying the hypotheses of Theorem~\ref{Main theorem}.
Let $\epsilon \in\{0,1 \}$ be such that $\chi(-1)= (-1)^\epsilon$. For $k\ge4$, $\Re(s)\in(\frac12,\frac{k-1}{2})$, $x>0$, and $\sigma_1\in(\frac{1-k}{2},-\Re(s))$, define 
\begin{equation}{\label{Main integral}}
    F_k(s,x)=\frac1{2\pi i}\int_{\Re(u)=\sigma_1}
\frac{\Gamma_\C(u+\frac{k-1}{2})\gamma(1-s-u)}{\Gamma_\C(-u+\frac{k+1}{2})\gamma(s+u)}
x^u\,du.
\end{equation}
By Stirling's formula, the integral converges absolutely since $\Re(s)>\frac12$.
Note also that our choice of $\sigma_1$ ensures that the contour $\Re(u)=\sigma_1$ 
separates the poles of $\Gamma_\C(u+\frac{k-1}{2})$ and $\gamma(1-s-u)$.

The proof will be split into two cases depending on whether $F_k(1,1)$ is nonzero for some $k$ or not.
\begin{prop}\label{prop1}
Suppose that $F_k(1,1)\ne0$ for some $k\ge4$ with $k\equiv\epsilon\pmod{2}$. Then $f$ as defined in Theorem~\ref{Main theorem} is in $S_k(N,\chi)$.
\end{prop}

\begin{proof}
Fix $k$ as in the hypotheses. For $n\in\Z_{>0}$ and $\Re(s)>1$, define 
\begin{equation}\label{e:Kn_def}
K_n(s,f,\chi) = \zeta^{(N)}(2s)\sum_{m=1}^\infty\frac{f_m}{m^s} \frac{ \Gamma(k-1)}{(4\pi)^{k-1}}\sum_{g\in H_k(N,\chi)}
\rho_g(n)\overline{\rho_g(m)} 
= \frac{ \Gamma(k-1)}{( 4 \pi)^{k-1}} \sum_{g\in H_k(N,\chi)}\rho_g(n)L(s,f\times\bar{g}), 
\end{equation}
where 
\[
L(s, f\times\bar{g}) = \zeta^{(N)}(2s) \sum_{m=1}^\infty \frac{f_m\overline{\rho_g(m)}}{m^s}
\quad\text{and}\quad\zeta^{(N)}(s)= \zeta(s)\prod_{p\mid N} \bigl(1 - p^{-s}\bigr).
\]
Note that the series for $L(s,f\times\bar{g})$ converges absolutely for $\Re(s)>1$ thanks to the Ramanujan bound $\rho_g(m)\ll_\varepsilon m^\varepsilon$.

Now consider $s\in\C$ with $\Re(s)\in(\frac54,\frac{k-1}{2})$.
Applying \eqref{e:Petersson} to  \eqref{e:Kn_def}, we get
\begin{align*}
K_{n}(s, f, \chi) &=   \zeta^{(N)}(2s)  \sum_{m=1}^\infty \frac{f_m}{m^{s}} \bigg\{ \delta_{n=m} + 2\pi i^{-k} \sum_{\substack{q\geq 1\\ N\mid q}} \frac{S_{\chi}(m, n; q)}{q} J_{k-1}\!\left(\frac{4\pi \sqrt{mn}}{q}\right)\bigg\}
\\ &=\zeta^{(N)}(2s)  \frac{f_n}{n^s}
+ 2\pi i^{-k}  \zeta^{(N)}(2s) 
\sum_{\substack{q\geq 1\\ N\mid q}} \frac{1}{q} \sum_{m=1}^\infty \frac{f_m S_{\chi}(m, n; q)}{m^s} J_{k-1}\!\left(\frac{4\pi \sqrt{mn}}{q}\right). 
\end{align*}
Here the interchange of the sums is justified for $\Re(s)>\frac54$ by the estimates
\[
J_{k-1}(y)\ll\min\bigl\{y^{k-1},y^{-1/2}\bigr\}
\quad\text{and}\quad
S_\chi(m,n;q)\ll_{n,\varepsilon}q^{1/2+\varepsilon}.
\]

Let 
\[
K_{0,n}(s,f,\chi) = K_n(s,f,\chi)-f_nn^{-s}\zeta^{(N)}(2s).
\]
By the Mellin--Barnes type integral representation \cite[(6.422)]{gr15}, we have
\[
2\pi J_{k-1}(4\pi y) = \frac{1}{2\pi i} \int_{\Re(u)=\sigma_0} \frac{\Gamma_\C(u+\frac{k-1}{2})}{\Gamma_\C(-u+\frac{k+1}{2})}y^{-2u}\,du,
\]
for any choice of $\sigma_0\in(\frac{1-k}{2},0)$.
Applying this with $\sigma_0\in(1-\Re(s),0)$, we may change the order of sum and integral, to obtain
\[
K_{0,n}(s, f , \chi) = i^{-k}  \zeta^{(N)}(2s)
\sum_{\substack{q\geq 1\\ N\mid q}} \frac{1}{2\pi i} \int_{\Re(u)=\sigma_0} \frac{\Gamma_\C(u+\frac{k-1}{2})}{\Gamma_\C(-u+\frac{k+1}{2})}q^{2u-1} 
\sum_{m=1}^\infty \frac{f_m S_{\chi}(m, n; q)}{m^s}
(mn)^{-u}\,du.
\]
Opening up the Kloosterman sum, this becomes
\[
\sum_{\substack{q\geq 1\\ N\mid q}} \frac{1}{2\pi i} \int_{\Re(u)=\sigma_0} \frac{\Gamma_\C(u+\frac{k-1}{2})}{\Gamma_\C(-u+\frac{k+1}{2})}q^{2u-1}n^{-u} 
\sum_{\substack{a\bmod{q}\\ a\bar{a}\equiv 1\bmod{q}}} \chi(a) e(n\bar{a}/q) 
L_f\bigl(s+u,\tfrac{a}{q}\bigr)\,du.
\]

Next we shift the contour to $\Re(u)=\sigma_1\in(\frac{1-k}{2},-\Re(s))$,
so that $\Re(1-s-u)>1$, and apply the functional equation 
\begin{align*}
\sum_{\substack{a\bmod{q}\\ a\bar{a}\equiv 1\bmod{q}}}\chi(a)e(n\bar{a}/q) 
L_f\bigl(s+u,\tfrac{a}{q}\bigr)=\omega q^{1-2s - 2u}\frac{\gamma(1-s-u)}{\gamma(s+u)} \sum_{\substack{a\bmod{q}\\ a\bar{a}\equiv 1\bmod{q}}}    e(n\bar{a}/q)L_f\bigl(1-s-u,-\tfrac{\bar{a}}{q}\bigr),
\end{align*}
obtaining
\begin{align*}
K_{0,n}(s, f, \chi) =i^{-k}\omega 
\zeta^{(N)}(2s)
\sum_{\substack{q\geq 1\\ N\mid q}} \frac{1}{q^{2s}}
\frac{1}{2\pi i}&\int_{\Re(u)=\sigma_1} \frac{\Gamma_\C(u+\frac{k-1}{2})}{\Gamma_\C(-u+\frac{k+1}{2})} \frac{\gamma(1-s-u)}{\gamma(s+u)}n^{-u} \\
    &\cdot\sum_{\substack{a\bmod{q}\\ a\bar{a}\equiv 1\bmod{q}}} 
e(n\bar{a}/q)L_f\bigl(1-s-u,-\tfrac{\bar{a}}{q}\bigr)\,du  . 
\end{align*}
Note that the contour shift is justified by the fact that $\Lambda_f\bigl(s,\tfrac{a}{q}\bigr)$ has finite order and the estimates
\[
L_f\bigl(1-s-u,-\tfrac{\bar{a}}{q}\bigr)\ll1
\quad\text{and}\quad
\frac{\Gamma_\C(u+\frac{k-1}{2})}{\Gamma_\C(-u+\frac{k+1}{2})} \frac{\gamma(1-s-u)}{\gamma(s+u)}
\ll|u|^{-2\Re(s)}
\quad\text{for }\Re(u)=\sigma_1.
\]

The same estimates show that we may swap the order of sum and integral. We also expand $L_f\bigl(1-s-u,-\tfrac{\bar{a}}{q}\bigr)$ as a Dirichlet series, obtaining 
\[
\sum_{\substack{a\bmod{q}\\ a\bar{a}\equiv 1\bmod{q}}} 
e(n\bar{a}/q)L_f\bigl(1-s-u,-\tfrac{\bar{a}}{q}\bigr)
=\sum_{\substack{a\bmod{q}\\ a\bar{a}\equiv 1\bmod{q}}} 
e(n\bar{a}/q)\sum_{m=1}^\infty
\frac{f_me(-m\bar{a}/q)}{m^{1-s-u}}
= \sum_{m=1}^\infty\frac{f_mr_q(n-m)}{m^{1-s-u}},
\]
where $r_q$ is the Ramanujan sum. An application of Lemma~\ref{lem:Ramanujansum_series} leads to the following expression.
\begin{equation}\label{final expression}
\begin{aligned}
K_n(s, f&,\chi) = f_nn^{-s}\zeta^{(N)}(2s)\\
&+ i^{-k}\omega f_nn^{s-1}
\zeta(2s-1) N^{1-2s} \prod_{p\mid N}(1-p^{-1})
\cdot\int_{\Re(u)=\sigma_1} \frac{\Gamma_\C(u+\frac{k-1}{2})  }{\Gamma_\C(-u+\frac{k+1}{2})} \frac{\gamma(1-s-u)}{\gamma(s+u)} \, du \\
\\ &+ i^{-k} \omega 
\sum_{\substack{m\geq 1\\ m\neq n}} \frac{f_m\sigma_{1-2s}(n-m; N)}{m^{1-s}}
   \int_{\Re(u)=\sigma_1} \frac{\Gamma_\C(u+\frac{k-1}{2})  }{\Gamma_\C(-u+\frac{k+1}{2})} \frac{\gamma(1-s-u)}{\gamma(s+u)} \left( \frac{m}{n}    \right)^u \, du .
\end{aligned}
\end{equation}
It is straightforward to see that $\sigma_{1-2s}(r;N)\ll_{N,\varepsilon}|r|^\varepsilon$, uniformly for $r\ne1$ and $\Re(s)\ge\frac12$. Thus, for a fixed $\sigma_1$, both integrals and the sum over $m$ converge absolutely for $\frac12<\Re(s)<-\sigma_1$.
This establishes the meromorphic continuation of $K_n(s,f,\chi)$ to that region, and hence also of
$\sum_{g\in H_k(N,\chi)}\rho_g(n)L(s,f\times\bar{g})$, in view of \eqref{e:Kn_def}.

Since the $g$ form a basis for $S_k(N,\chi)$, we can choose a finite set $\{n_i:i=1,\ldots,d\}$, where $d=\dim S_k(N,\chi)$, such that the vectors $(\rho_g(n_1),\ldots,\rho_g(n_d))$ are linearly independent. Taking a suitable linear combination of \eqref{final expression} for $n=n_i$, we deduce the meromorphic continuation of $L(s,f\times\bar{g})$ to $\Re(s)>\frac12$ for each individual $g$. The only possible pole is at $s=1$, and taking residues we see that
\[
\frac{\Gamma(k-1)}{(4\pi)^{k-1}}\sum_{g\in H_k(N, \chi)}
\rho_g(n)\Res_{s=1}L(s,f\times\bar{g}) = \tfrac12i^{-k}\omega N^{-1} \prod_{p\mid N}(1-p^{-1})\cdot F_k(1,1)f_n. 
\]
Since $F_k(1,1)\ne0$, we see that there exist $x_g\in\C$ such that
$f_n = \sum_{g\in H_k(N, \chi)} x_g\rho_g(n)$
for any positive integer $n$. Since $S_k(N,\chi)$ is a vector space, we have proved the claim. 
\end{proof}
We have taken care of the case where the integral \eqref{Main integral} is nonzero at $s=1$, $x=1$ for some $k\ge4$. If this is not the case, we use the following proposition. 
\begin{prop}\label{prop2}
If $F_k(1,1)=0$ for all $ k \ge4$ with $k\equiv\epsilon\pmod{2}$, then $\gamma(s)$ is of the form $cH^s\Gamma_\C(s+\frac{\ell-1}{2})$, where $c,H\in\R_{>0}$, and $\ell\in\{1,2,3\}$ with $\ell\equiv\epsilon\pmod2$.
\end{prop}
\begin{proof}
We replace $u$ by $u/2$ in the definition of $F_k(1,1)$ and shift the contour to $\Re(u)=-\frac52$, which is permissible for all $k\ge4$. Our hypothesis then implies that
\begin{equation}\label{stuff}
\frac1{2\pi i}\int_{\Re(u)=-\frac52}
\frac{\Gamma_\C(\frac{k-1+u}{2})\gamma(-u/2)}{2 \Gamma_\C(\frac{k+1-u}{2})\gamma(1+u/2)}\,du
=0
\quad\text{for all }k\ge4\text{ with }k\equiv\epsilon\;(\text{mod }2).
\end{equation}
For $n\ge0$, define
\[
f_n(y)=\frac{\mathbf{1}_{(0,1)}(y)}{\sqrt{1-y^2}}\begin{cases}
\cos(n\arcsin{y})&\text{if }2\mid n,\\
\sin(n\arcsin{y})&\text{if }2\nmid n.
\end{cases}
\]
Using the formulas in \cite[\S3.631]{gr15}, we see that $f_n$ has Mellin transform
\[
\widetilde{f}_n(s)=
\int_0^\infty f_n(y)y^{s-1}\,dy
=\frac{(-1)^{\lfloor{n/2}\rfloor}\Gamma_\C(s)}
{2^s\Gamma_\C(\frac{s+n+1}{2})\Gamma_\C(\frac{s-n+1}{2})}
\quad\text{for }\Re(s)>0.
\]
For $k \equiv \epsilon \pmod{2}$ we have
\begin{align*}
\frac{\Gamma_\C(\frac{k-1+u}{2})\gamma(-u/2)}{2\Gamma_\C(\frac{k+1-u}{2})\gamma(1+u/2)}
&=\frac{\Gamma_\C(1-u)}{2^{1-u}\Gamma_\C(\frac{k+1-u}{2})\Gamma_\C(\frac{3-k-u}{2})}
\cdot\frac{2^{-u}\Gamma_\C(\frac{3-k-u}{2})\Gamma_\C(\frac{k-1+u}{2})}{\Gamma_\C(1-u)}
\frac{\gamma(-u/2)}{\gamma(1+u/2)}\\
&=\frac{(-1)^{\lfloor{(k-1)/2}\rfloor}\widetilde{f}_{k-1}(1-u)}{2^{u-1}\Gamma_\C(1-u)\sin(\frac{\pi}{2}(u+k-1))}
\frac{\gamma(-u/2)}{\gamma(1+u/2)}\\
&=\frac{\widetilde{f}_{k-1}(1-u)}{2^{u-1}\Gamma_\C(1-u)\sin(\frac{\pi}{2}(u+1-\epsilon))}
\frac{\gamma(-u/2)}{\gamma(1+u/2)}\\
&=\widetilde{f}_{k-1}(1-u)
\frac{\sqrt2\Gamma_\C(\frac{1-\epsilon+u}{2})}
{\Gamma_\C(\frac{2-\epsilon-u}{2})}
\frac{\gamma(-u/2)}{\gamma(1+u/2)}.
\end{align*}
Thus, \eqref{stuff} can be written as
\[
\frac1{2\pi i}\int_{\Re(u)=-\frac52}\widetilde{f}_{k-1}(1-u)\widetilde{g}(u)\,du=0,
\quad\text{where}\quad
\widetilde{g}(u)=
\frac{\sqrt{2}\Gamma_\C(\frac{1-\epsilon+u}{2})}
{\Gamma_\C(\frac{2-\epsilon-u}{2})}
\frac{\gamma(-u/2)}{\gamma(1+u/2)}.
\]

Applying the inverse Mellin transform, we define 
\[
g(y) = \frac{1}{2\pi i}
\int_{\Re(u)=-\frac52}\widetilde{g}(u)y^{-u}\,du
\quad\text{for }y>0.
\]
By Stirling's formula we have $\widetilde{g}(u)\ll|u|^{-3/2}$ for
$\Re(u)=-\frac52$,
and it follows that $g(y)$ extends continuously to $[0,\infty)$.

An application of Fubini's theorem shows that for any measurable functions $f,g$ on $(0,\infty)$ satisfying  $\int_0^\infty|f(y)|y^{-\sigma}\,dy<\infty$ and
$\int_\R|\widetilde{g}(\sigma+it)|\,dt<\infty$, we have
\[
\frac1{2\pi i}\int_{\Re(u)=\sigma}\widetilde{f}(1-u)\widetilde{g}(u)\,du
=\int_0^\infty f(y)g(y)\,dy.
\]
It is easy to see this is satisfied by the functions $f_n$ and $g$ above with $\sigma=-\frac52$. Hence, \eqref{stuff} becomes
\[
\int_0^1f_{k-1}(y)g(y)\,dy=0
\quad\text{for all }k\ge4\text{ with }k\equiv\epsilon\;(\text{mod }2).
\]

Suppose $\epsilon=1$. Then we have
\[
\int_{-\pi}^{\pi}
\cos(\tfrac{k-1}{2}\theta)g(|\sin(\theta/2)|)\,d\theta=0
\quad\text{for every odd }k\ge5.
\]
Thus, the function $h(\theta)=g(|\sin(\theta/2)|)$ has Fourier series of the form $a+b\cos\theta$. Since $h$ is continuous, we have $h(\theta)=a+b\cos\theta$ for all $\theta$, and thus $g(y)=a+b-2by^2$ for $y\in[0,1]$.

Since $g(y)=O(y^{5/2})$, we must have $a=b=0$.
Computing the Mellin transform again, we see that
$\widetilde{g}(u)=\int_1^\infty g(y)\,y^{u-1}\,dy$,
so $\widetilde{g}(u)$ is analytic for $\Re(u)<-\frac52$. 
Since $\frac{\gamma(-u/2)}{\Gamma_\C(\frac{1-u}{2})}$ is analytic and nonvanishing for $\Re(u)<-\frac52$, it follows that $\frac{\Gamma_\C(\frac{u}{2})}{\gamma(1+u/2)}$ is analytic for $\Re(u)<-\frac52$. This means that $\gamma(s)$ has a pole at each negative integer.

Applying Lemma~\ref{no pole lemma}, we have $\gamma(s)=cP(s)H^s\Gamma_\C(s)$, where $P$ is a monic polynomial whose roots are distinct non-positive integers. Since $\gamma(s)$ has poles at all negative integers, either $P(s)=1$ or $P(s)=s$. Thus, $\gamma(s)$ is of the form $c'H^s\Gamma_\C(s+\frac{\ell-1}{2})$ for some $\ell\in\{1,3\}$, as required.

\medskip
Now suppose that $\epsilon=0$. Then
\begin{equation}\label{Cheb}
\int_{-\pi}^\pi\frac{\sin(\frac{k-1}{2}\theta)}{\sin(\theta/2)}
|\sin(\theta/2)|g(|\sin(\theta/2)|)\,d\theta=0
\quad\text{for every even }k\ge4.
\end{equation}
Note that
$
\frac{\sin(\frac{k-1}{2}\theta)}{\sin(\theta/2)}
=W_{\frac{k-2}{2}}(\cos\theta),
$
where $W_n$ is a polynomial of degree $n$ (the `Chebyshev polynomial of the fourth kind', see \cite{m93}).
Writing $v=\cos\theta=1-2\sin^2(\theta/2)$, \eqref{Cheb} becomes
\[  0=\int_{-1}^1 W_{\frac{k-2}{2}}(v)h(v)
\sqrt{\frac{1-v}{1+v}}\,dv,
\quad\text{where }
h(v)=\frac{g\!\left(\sqrt{\frac{1-v}{2}}\right)}{\sqrt{1-v}}
\text{ for }v\in(-1,1).
\]
Since the $W_n$ are an orthogonal family with respect to the measure $\sqrt{\frac{1-v}{1+v}}\,dv$, there exists a constant $a$ such that $(h(v)-a)\sqrt{\frac{1-v}{1+v}}$ is continuous and absolutely integrable on $(-1,1)$, and orthogonal to all polynomials. It follows that $h(v)=a$ for all $v\in(-1,1)$, and thus $g(y)=a\sqrt{2}y$ for $y\in(0,1)$.
Since $g(y)=O(y^{5/2})$, we must have $a=0$, and thus
$\widetilde{g}(u)=\int_1^\infty g(y)\,y^{u-1}\,dy$.

As before we conclude that $\frac{\Gamma_\C(\frac{1+u}{2})}{\gamma(1+u/2)}$ is analytic for $\Re(u)<-\frac52$. Defining $\tilde{\gamma}(s)=\gamma(s+\frac12)$ we see that $\tilde{\gamma}(s)$ satisfies the hypotheses imposed on $\gamma(s)$ in Theorem~\ref{Main theorem} and has a pole at every negative integer. Appealing again to Lemma~\ref{no pole lemma}, we find that $\tilde{\gamma}(s)$ is of the form $cP(s)H^s\Gamma_\C(s)$, where either $P(s)=1$ or $P(s)=s$.  
Thus $\gamma(s)=c'H^sP(s-\frac12)\Gamma_\C(s-\frac12)$.
Thanks to the hypothesis $\Re(\mu_j)>-\frac12\lambda_j$ in Theorem~\ref{Main theorem}, $\gamma(s)$ cannnot have a pole at $\frac12$, and therefore $P(s)=s$. Hence $\gamma(s)=c''H^s\Gamma_\C(s+\frac12)$, as required.
\end{proof}

Now we can complete the proof of Theorem~\ref{Main theorem}. 
In view of Propositions~\ref{prop1} and \ref{prop2}, we may assume that $\gamma(s)=cH^s\Gamma_\C(s+\frac{\ell-1}{2})$, where $c,H\in\R_{>0}$ and $\ell\in\{1,2,3\}$ with $\ell\equiv\epsilon\pmod2$. 
In this case we fall back on a more traditional proof of the converse theorem as in \cite{r77}, but we must first address the fact that our gamma factor differs from the expected one by the exponential factor $H^s$.

Suppose first that $H>1$. Equation \eqref{stuff} becomes 
\begin{equation*}
\frac1{2\pi i}\int_{\Re(u)=-\frac52} H^{-u}
\frac{\Gamma_\C(\frac{k-1+u}{2})\Gamma_\C(\frac{-u}{2}+\frac{\ell-1}{2})}{\Gamma_\C(\frac{k+1-u}{2})\Gamma_\C(1+\frac{u}{2}+\frac{\ell-1}{2})}\,du
=0
\quad\text{for all }k\ge4\text{ with }k\equiv\epsilon\;(\text{mod }2).
\end{equation*}
Since $H>1$, we can shift the contour to the right, as the integrand vanishes in the limit as $\Re(u)\rightarrow\infty$.

Suppose $\ell=2$. Then when $k=4$, we pick up poles at $u=1,3$ and derive that
$0= H^{-1}-H^{-3}$.
Thus $H=1$, giving a contradiction. Similarly, when $\ell=3$ we take $k=5$; we have poles at $u=2,4$, getting
$0=H^{-2} - H^{-4}$,
which results in the same contradiction. When $\ell=1$, 
looking at $k=5$, we have poles at $u=0,2,4$; thus 
$0= 1- 4 H^{-2}  + 3H^{-4}$,
which implies that $H^{-2}=\frac{1}{3}$. Now looking at $k=7$, we get an extra pole at $u=6$, so that
$0= \frac{ 2}{ 3 }- 6 H^{-2} + 12 H^{-4} - \frac{20}{3}H^{-6}$,
which is not satisfied for $H^{-2}= \frac{1}{3}$.

Hence $H\le1$.
Let $f(z)= \sum_{n=1}^{\infty} f_nn^{ \frac{\ell-1}{2}}e(nz)$.
Applying Hecke's argument \cite[Theorem~4.3.5]{m06} to our Voronoi formulas \eqref{functional equation}, we get the modularity relation
\begin{equation}{\label{modularity original}}
    f\!\left(\frac{ \frac{-1}{H^2z}+a}{q}  \right)
     = \omega\chi(\bar{a})(-iHz)^\ell f\!\left(\frac{z-\bar{a}}{q}\right)
\end{equation} 
for all $q\in N\Z_{>0}$ and $a,\bar{a}\in\Z$ with $a\bar{a}\equiv1\pmod{q}$. 

The $\ell$-slash operator is defined for matrices $M=\begin{psmallmatrix}
 a & b \\
 c & d 
\end{psmallmatrix}$ of positive determinant by 
\[  (f|M)(z)= (\det M)^{\ell/2} (cz+d)^{-\ell} f\!\left(\frac{az+b}{cz+d}\right). \]
In this notation, \eqref{modularity original} becomes 
\[     f\bigg| \begin{pmatrix}
aH^2    &    -1 \\
qH^2 & 0   \\
\end{pmatrix} = i^{-\ell}\omega\chi(\bar{a})f\bigg| \begin{pmatrix}
1   &    -\bar{a} \\
0 & q  \\
\end{pmatrix}, \]
or equivalently
\begin{equation}{\label{Modularity for H}}
      f \bigg| \begin{pmatrix}
 aH & \frac{a\bar{a} H^2-1}{qH}  \\
 qH &  \bar{a} H  \\
\end{pmatrix} = i^{-\ell}\omega\chi(\bar{a})f  . 
\end{equation}                      

We may assume that $f$ is not identically $0$, since the conclusion is trivial otherwise.
Applying \eqref{functional equation} twice, we have
\[
\Lambda_f\bigl(s,\tfrac{a}{q}\bigr) = \omega\chi(\bar{a})q^{1-2s}\Lambda_f\bigl(1-s,-\tfrac{\bar{a}}{q}\bigr)
=\omega^2\chi(-1)\Lambda_f\bigl(s,\tfrac{a}{q}\bigr),
\]
and thus $(i^{-\ell}\omega)^2=(-1)^\ell\chi(-1)=1$.

Suppose $H<1$. Taking $a=\bar{a}=1$, the matrix above has absolute trace $2H<2$, so is elliptic. Unless $2H=\zeta +\zeta^{-1}$ for some root of unity $\zeta$, this matrix has infinite order, and then a generalization of Weil's lemma \cite[Lemma~4.2]{src18} implies that $f=0$.

Hence $2H$ must be an algebraic integer. For prime $p$ and $i=1,2$, let 
$M_{i,p}$ be the matrix $\begin{pmatrix}
  H & \frac{H-H^{-1}}{q_i} \\
  q_i H & H 
\end{pmatrix}$, where $q_1=pN$, $q_2=(p+1)N$. 
We compute that
\[   \tr(M_{1,p}M_{2,p})=(2H)^2-2-\bigl(4-(2H)^2\bigr)\bigl(4p(p+1)\bigr)^{-1}. \]
Since $\lim_{p\rightarrow\infty}\tr(M_{1,p}M_{2,p})=(2H)^2-2$, we have $|\tr(M_{1,p}M_{2,p})|<2$ for all sufficiently large primes $p$, and
again by Weil's lemma, it follows that $\tr(M_{1,p}M_{2,p})$ is an algebraic integer.
Let $K$ be a number field containing $2H$. Clearly $\tr(M_{1,p}M_{2,p})\in K$ for each $p$. Taking norms,
\[ N_{K/\Q}\bigl(\tr(M_{1,p}M_{2,p})+2-(2H)^2\bigr) = \bigl(-4p(p+1)\bigr)^{-[K:\Q]} N_{K/\Q}\bigl(4-(2H)^2\bigr). \]
Taking a sufficiently large prime $p$, this is not an integer, giving a contradiction. 

Hence $H=1$. By \eqref{Modularity for H}, it follows that
\[ f|M = i^{-\ell}\omega\chi(M)f
\quad\text{for }
M=\begin{pmatrix}
 a & b \\
 q & \bar{a} 
\end{pmatrix} \in \Gamma_0(N)
\text{ with }q>0,
\]
where we define $\chi(M)= \chi(\bar{a})$. Taking $M_1,M_2\in\Gamma_0(N)$ such that $M_1$, $M_2$ and $M_1M_2$ are all of this form, we have 
\[ i^{-\ell}\omega\chi(M_1M_2)f=f|M_1M_2=(f|M_1)|M_2 = (i^{-\ell}\omega)^2\chi(M_1)\chi(M_2)f.\]
Since $\chi(M_1M_2)=\chi(M_1)\chi(M_2)$, we have must have $\omega=i^\ell$. 
Finally, since the $M$ of the above form, together with $ \begin{psmallmatrix}
 1 & 1 \\
 0 & 1 
\end{psmallmatrix}$ and $\begin{psmallmatrix}
 -1 & 0 \\
 0 & -1 
\end{psmallmatrix}$, generate $\Gamma_0(N)$, we conclude that $f\in M_\ell(N,\chi)$, as required.

\thispagestyle{empty}
\bibliographystyle{amsalpha}
\bibliography{reference}

\newcommand{\etalchar}[1]{$^{#1}$}
\providecommand{\bysame}{\leavevmode\hbox to3em{\hrulefill}\thinspace}
\providecommand{\MR}{\relax\ifhmode\unskip\space\fi MR }
% \MRhref is called by the amsart/book/proc definition of \MR.
\providecommand{\MRhref}[2]{%
  \href{http://www.ams.org/mathscinet-getitem?mr=#1}{#2}
}
\providecommand{\href}[2]{#2}
\begin{thebibliography}{BBB{\etalchar{+}}18}

\bibitem[BBB{\etalchar{+}}18]{src18}
Sandro Bettin, Jonathan~W. Bober, Andrew~R. Booker, Brian Conrey, Min Lee,
  Giuseppe Molteni, Thomas Oliver, David~J. Platt, and Raphael~S. Steiner,
  \emph{A conjectural extension of {H}ecke's converse theorem}, Ramanujan J.
  \textbf{47} (2018), no.~3, 659--684. \MR{3874812}

\bibitem[BK14]{bk14}
Andrew~R. Booker and M.~Krishnamurthy, \emph{Weil's converse theorem with
  poles}, Int. Math. Res. Not. IMRN (2014), no.~19, 5328--5339. \MR{3267373}

\bibitem[GR15]{gr15}
I.~S. Gradshteyn and I.~M. Ryzhik, \emph{Table of integrals, series, and
  products}, eighth ed., Elsevier/Academic Press, Amsterdam, 2015, Translated
  from the Russian, Translation edited and with a preface by Daniel Zwillinger
  and Victor Moll, Revised from the seventh edition [MR2360010]. \MR{3307944}

\bibitem[HLN21]{m21}
Jeff Hoffstein, Min Lee, and Maria Nastasescu, \emph{First moments of
  {R}ankin-{S}elberg convolutions of automorphic forms on {$\rm GL(2)$}}, Res.
  Number Theory \textbf{7} (2021), no.~4, Paper No. 60, 44. \MR{4314222}

\bibitem[IK04]{ik04}
Henryk Iwaniec and Emmanuel Kowalski, \emph{Analytic number theory}, American
  Mathematical Society Colloquium Publications, vol.~53, American Mathematical
  Society, Providence, RI, 2004. \MR{2061214}

\bibitem[KMV02]{kmv02}
E.~Kowalski, P.~Michel, and J.~VanderKam, \emph{Rankin-{S}elberg
  {$L$}-functions in the level aspect}, Duke Math. J. \textbf{114} (2002),
  no.~1, 123--191. \MR{1915038}

\bibitem[KP15]{kp15}
Jerzy Kaczorowski and Alberto Perelli, \emph{Twists, {E}uler products and a
  converse theorem for {$L$}-functions of degree 2}, Ann. Sc. Norm. Super. Pisa
  Cl. Sci. (5) \textbf{14} (2015), no.~2, 441--480. \MR{3410616}

\bibitem[KP20]{kp20}
J.~Kaczorowski and A.~Perelli, \emph{Classification of {$L$}-functions of
  degree 2 and conductor 1}, arXiv preprint arXiv:2009.12329 (2020).

\bibitem[Mas93]{m93}
J.~C. Mason, \emph{Chebyshev polynomials of the second, third and fourth kinds
  in approximation, indefinite integration, and integral transforms},
  Proceedings of the {S}eventh {S}panish {S}ymposium on {O}rthogonal
  {P}olynomials and {A}pplications ({VII} {SPOA}) ({G}ranada, 1991), vol.~49,
  1993, pp.~169--178. \MR{1256024}

\bibitem[Miy06]{m06}
Toshitsune Miyake, \emph{Modular forms}, english ed., Springer Monographs in
  Mathematics, Springer-Verlag, Berlin, 2006, Translated from the 1976 Japanese
  original by Yoshitaka Maeda. \MR{2194815}

\bibitem[Raz77]{r77}
Michael~J. Razar, \emph{Modular forms for {$G_{0}(N)$} and {D}irichlet series},
  Trans. Amer. Math. Soc. \textbf{231} (1977), no.~2, 489--495. \MR{444576}

\bibitem[Ven02]{akshay}
Akshay Venkatesh, \emph{Limiting forms of the trace formula}, ProQuest LLC, Ann
  Arbor, MI, 2002, Thesis (Ph.D.)--Princeton University. \MR{2703729}

\end{thebibliography}
\end{document}